\documentclass[12pt,fleqn]{article}

\usepackage{amsmath}
\usepackage{amsthm}
\usepackage{amssymb}
\usepackage{amsfonts}
\usepackage{epsf}
\usepackage{graphicx}
\usepackage{epstopdf}
\usepackage{hyperref}
\usepackage{color}
\usepackage{subfigure}

\newtheorem{lemma}{Lemma}
\newtheorem{thm}{Theorem}

\newtheorem{prop}{Proposition}

\theoremstyle{remark}
\newtheorem{rmk}{Remark}
\theoremstyle{definition}

 \DeclareMathOperator\Ai{{Ai}}
 
 \numberwithin{equation}{section}

\numberwithin{equation}{section}

\newcounter{comment}
\def\a{\alpha}

\def\sgn{\textrm{sgn}}

\def\AS{\textrm{AS}}

\begin{document}

\title{On the quasi-Ablowitz-Segur and quasi-Hastings-McLeod solutions of the inhomogeneous Painlev\'{e} II equation}
\date{\today}
\author{Dan Dai$^{\ast}$ and Weiying Hu$^{\dag}$}

\maketitle

\begin{abstract}
   We consider the quasi-Ablowitz-Segur and quasi-Hastings-McLeod solutions of the inhomogeneous Painlev\'{e} II equation
   $$
   u''(x)=2u^3(x)+xu(x)-\alpha \qquad \textrm{for } \a \in \mathbb{R} \textrm{ and } |\a| > \frac{1}{2}.
   $$
   These solutions are obtained from the classical Ablowitz-Segur and Hastings-McLeod solutions via the B\"{a}cklund transformation, and satisfy the same asymptotic behaviors when $x \to \pm \infty$. For $|\a| > 1/2$, we show that the quasi-Ablowitz-Segur and quasi-Hastings-McLeod solutions possess $ [ \, |\alpha| + \frac{1}{2} \, ] $ simple poles on the real axis, which rigorously justifies the numerical results in Fornberg and  Weideman (\emph{Found. Comput. Math.}, \textbf{14} (2014), no. 5, 985--1016).
\end{abstract}


\vspace{2cm}




\vspace{5mm}

\hrule width 65mm

\vspace{2mm}

\begin{description}

\item \hspace*{5mm}$\ast$ Department of Mathematics, City University of
Hong Kong, Hong Kong. \\
Email: \texttt{dandai@cityu.edu.hk}

\item \hspace*{5mm}$\dag$ Department of Mathematics, City University of
Hong Kong, Hong Kong. \\
Email: \texttt{weiyinghu2-c@my.cityu.edu.hk} (corresponding author)

\end{description}

\newpage

\section{Introduction and statement of results}

\subsection{Ablowitz-Segur and Hastings-McLeod solutions}

We consider the following \emph{inhomogeneous} Painlev\'e II equation (PII)
\begin{equation}\label{PII-def}
   u''(x)=2u^3+xu-\a, \qquad \a \in \mathbb{R} \setminus \{0\}.
\end{equation}
When $\a = 0$, the above equation is reduced to the \emph{homogeneous} PII. It is well-known that, PII possesses two families of special solutions which are real and pole-free on the real axis: one family of these solutions is oscillatory and bounded, namely the \emph{Ablowitz-Segur} (AS) solutions; the other family is smooth and nonoscillatory, namely the \emph{Hastings-McLeod} (HM) solutions. Both families of solutions decay like $\a/x$ as $ x \to +\infty$. More precisely, they have the following behaviors at $x \to \pm \infty$.

\medskip
\noindent \textbf{Ablowitz-Segur solutions: $\a \in (-1/2,1/2)$.}
\medskip

Let $k_{\a}$ be a real parameter and $k_{\a}\in (-\cos \pi \a, \cos \pi \a)$. The AS solution $u_{\textrm{AS}}(x;\a)$ is a one-parameter family of solutions of inhomogeneous PII \eqref{PII-def}, which is continuous on the real line and has the following asymptotic behaviors:
 \begin{eqnarray}
    u_{\textrm{AS}}(x;\a)&=& B(x; \alpha) + k_{\a}\Ai(x)(1+O(x^{-3/4})), \qquad \textrm{as} \,\, x \to +\infty,\label{asy-pos-AS}\\
    u_{\textrm{AS}}(x;\a)&=&\frac{d}{(-x)^{1/4}}\cos\{\frac{2}{3}(-x)^{3/2}-\frac{3}{4}d^2\ln(-x) + \phi\} +O(|x|^{-1}),    \label{asy-neg-AS} \\
   &&\hspace{6.5cm} \quad \textrm{as}\,\, x\to -\infty,\nonumber
    \end{eqnarray}
where $\Ai(x)$ is the Airy function,
\begin{equation} \label{series-B}
  B(x; \alpha)\sim \frac{\alpha}{x}\sum_{n=0}^{\infty} \frac{a_n}{x^{3n}}, \qquad a_0 = 1
\end{equation}
and $a_{n+1}=(3n+1)(3n+2)a_n - 2\a^2 \sum_{k,l,m =0}^{n} a_k a_l a_m.$
The constants $d$ and $\phi$ in \eqref{asy-pos-AS} and \eqref{asy-neg-AS} satisfy the following connection formulas
\begin{eqnarray}
    d(k_{\a}) &=& \frac{1}{\sqrt{\pi}} \sqrt{-\ln(\cos^2(\pi \alpha)-k_{\a}^2)}, \label{d-k}\\
    \phi(k_{\a}) &=& -\frac{3}{2} d^2 \ln 2 +\arg \Gamma{\biggr(\frac{1}{2}id^2}\biggr)- \frac{\pi}{4} -\arg (-\sin \pi \a -k_{\a}i). \label{phi-k}
\end{eqnarray}

\medskip
\noindent \textbf{Hastings-McLeod solutions: $\a \in \mathbb{R}$.}

\medskip

The HM solutions $u_{\textrm{HM}}(x;\a)$ of the inhomogeneous PII \eqref{PII-def} are continuous on the real axis and have the following asymptotic behaviors
\begin{eqnarray}
    u_{\textrm{HM}}(x;\a)&=& B(x; \alpha) + \sigma \cos(\pi \a) \Ai(x)(1+O(x^{-3/4})), \qquad \textrm{as} \,\, x \to +\infty,\label{asy-pos-HM} \\
    u_{\textrm{HM}}(x;\a)&=&\sigma \sqrt{\frac{-x}{2}} - \frac{\a}{2x}+O(-x)^{-3/2},\hspace{2.5cm} \qquad \textrm{as}\,\, x\to -\infty, \label{asy-neg-HM}
\end{eqnarray}
where $\sigma\in \{+1, -1\}$ and the series $B(x; \alpha)$ is given in \eqref{series-B}.
In \eqref{asy-pos-HM} and \eqref{asy-neg-HM}, the coefficient $\sigma$ depends on $\a$ as follows:
\begin{equation}\label{condition-HM-pole-free}
  \textrm{(i) if } \a > -1/2: \quad \sigma =+1; \qquad \textrm{(ii) if } \a < 1/2: \quad \sigma = -1.
\end{equation}

From the above formulas, one immediately sees that, when $\a \in (-\infty, -1/2] \cup [1/2, +\infty) $, there is a \emph{unique} HM solution for each $\a$; while when $\a \in (-\frac{1}{2}, \frac{1}{2})$, there exist \emph{two} HM solutions. Depending on whether they are monotonic on the whole real axis, the solutions can be separated into two families, which are called the \emph{primary Hastings-McLeod solutions} (pHM) and \emph{secondary Hastings-McLeod solutions} (sHM) in Fornberg and Weideman \cite{Fornberg2014}; see Figure \ref{phm-shm} for a sketch of their properties. They satisfy the asymptotics in \eqref{asy-pos-HM} and \eqref{asy-neg-HM} with the parameter $\sigma$ given by
\begin{figure}[h]
 \centering
 \subfigure{
   \includegraphics[width=6cm]{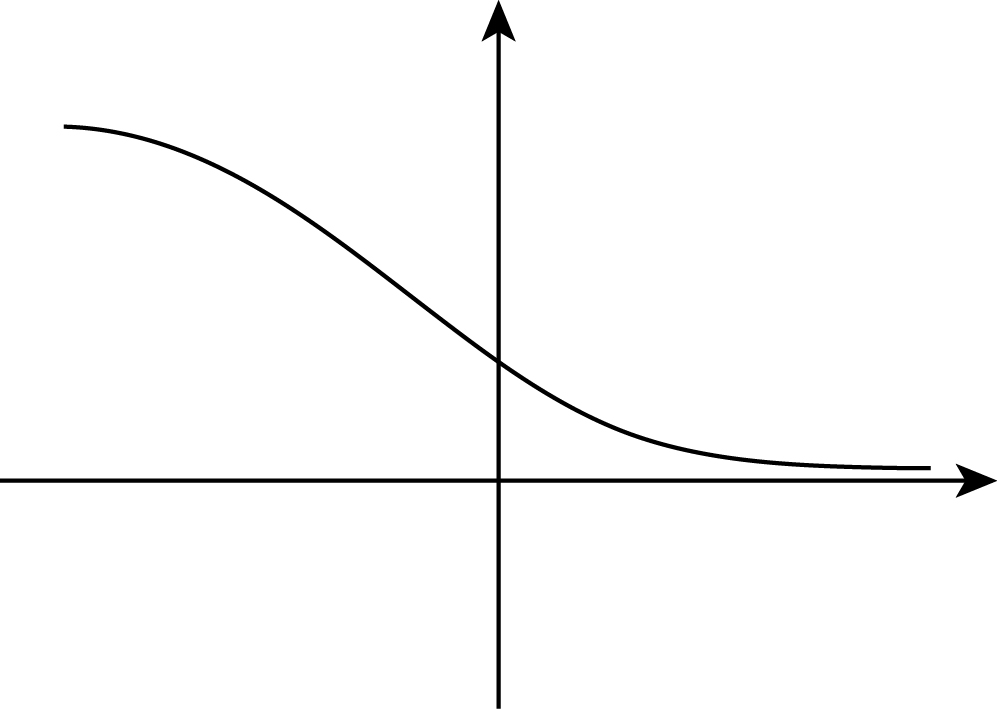}}
 \subfigure{
   \includegraphics[width=6cm]{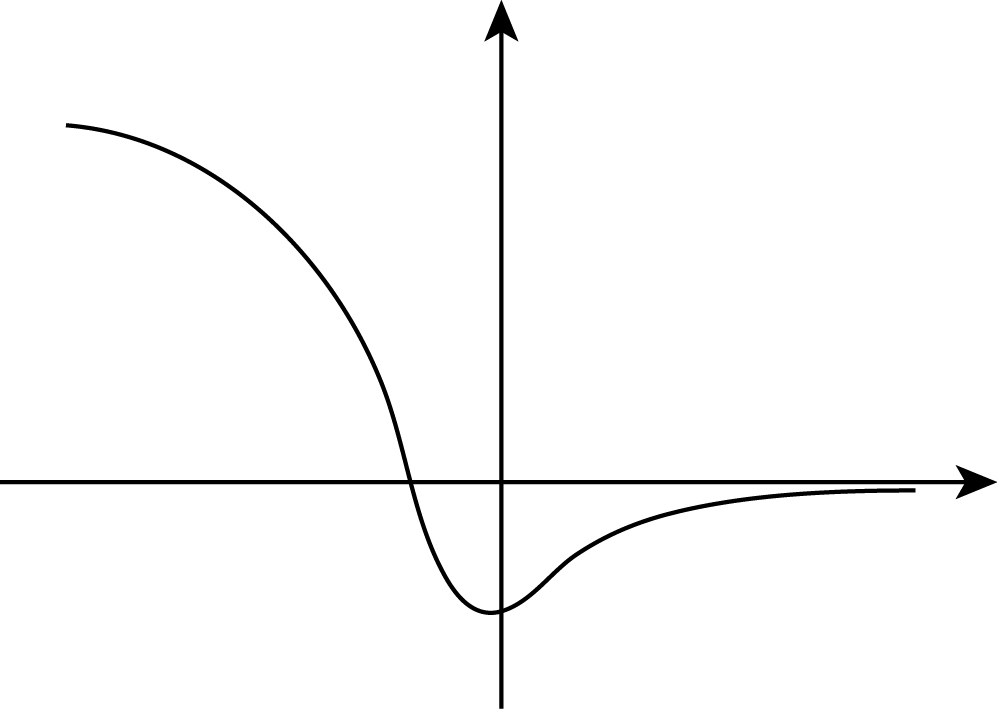}}
 \caption{The pHM (left) $u_{\textrm{pHM}}(x;\a)$ and sHM (right) $-u_{\textrm{sHM}}(x;\a)$ solutions of PII with the same parameter $0<\a <1/2$.}\label{phm-shm}
\end{figure}
\begin{align}
  \textrm{pHM (monotonic):} \qquad & \sigma = \begin{cases}
     \sgn(\a), & \textrm{if } \a \neq 0, \\ 1 \textrm{ or } -1, & \textrm{if } \a  = 0,
  \end{cases} \label{pHM-sigma-value} \\
  \textrm{sHM (not monotonic):} \qquad  & \sigma = -\sgn(\a), \qquad \textrm{if } \a \in (-1/2, 0) \cup (0, 1/2). \label{sHM-sigma-value}
\end{align}
From \eqref{asy-pos-HM}-\eqref{sHM-sigma-value}, one can see that the pHM solutions have same signs in their asymptotic behaviors, i.e., $u_{\textrm{pHM}}(x;\a) \sim \a/x$ as $x \to +\infty$ and $u_{\textrm{pHM}}(x;\a)\sim\sgn(\a) \sqrt{-x/2}$ as $x \to -\infty$ for $\a \neq 0$.
This is similar to the classical HM solution $u_{\textrm{HM}}(x;0)$ of the homogeneous PII whose asymptotic behaviors are $u_{\textrm{HM}}(x;0) \sim \Ai(x)$ as $x \to +\infty$ and $u_{\textrm{HM}}(x;\a)\sim \sqrt{-x/2}$ as $x \to -\infty$. It is well-known that the HM solution $u_{\textrm{HM}}(x;0)$ is monotonic on the real axis and possesses a unique inflexion point $x_0$ where $u_{\textrm{HM}}''(x_0;0) = 0$; see Hastings and McLeod \cite{HM1980}. From the numerical evidence in Fornberg and Weideman \cite[Fig. 10]{Fornberg2014}, the pHM solutions satisfy similar properties. For the sHM solutions, they have different signs in their asymptotics as $x \to \pm \infty$, and are no longer monotonic. Recently, these properties have been proved rigorously in Clerc et al. \cite{Clerc2017} and Troy \cite{Troy2018}.

The formulas \eqref{pHM-sigma-value} and \eqref{sHM-sigma-value} indicate that there exists a family of pHM solutions for any $\a \in \mathbb{R}$; and there is one additional family of sHM solutions for $\a \in (-1/2, 0) \cup (0, 1/2)$. This result is actually proved in Claeys, Kuijlaars and Vanlessen \cite[Theorem 1.1]{Claeys2008}. In \cite{Claeys2008}, the authors showed that, for $\a > - 1/2$, there exist the HM solutions which are pole-free on the real line and uniquely determined by the following asymptotic behaviors
\begin{equation*}
  u(x;\a) \sim \a /x , \quad  \textrm{as } x \to +\infty \quad \textrm{and} \quad u(x;\a) \sim \sqrt{-x/2} , \quad  \textrm{as } x \to -\infty.
\end{equation*}
With the classification in \eqref{pHM-sigma-value} and \eqref{sHM-sigma-value}, one can see that the above solutions are indeed the pHM and sHM solutions when $\a > 0$ and $-1/2<\a <0$, respectively. Using the following symmetry relation
\begin{equation}\label{sym-relation}
u(x;\a)=-u(x;-\a),
\end{equation}
one immediately gets the pHM and sHM solutions for $\a < 0$ and $0<\a <1/2$.

The AS and HM solutions for the homogeneous PII were first discovered by Ablowitz and Segur in \cite{Ablowitz1977-asymptotic,Segur1981} and Hastings and McLeod in  \cite{HM1980}, respectively. For the inhomogeneous PII, these solutions were obtained later by McCoy and Tang \cite{mccoy:Tang1986}, Its and Kapaev \cite{Its2003} and Kapaev \cite{Kapaev2004}. The rigorous justification of the asymptotic behaviors in \eqref{asy-pos-AS}-\eqref{asy-neg-AS} and \eqref{asy-pos-HM}-\eqref{asy-neg-HM}, as well as the connection formulas \eqref{d-k}-\eqref{phi-k}, has attracted a lot of research interest in the literature; see for example \cite{Bas:Cla:Law:McL1998,Clarkson1988,Deift1995,Kapaev1992} and the monograph by Fokas et al. \cite{Fokas2006}. All the AS and HM solutions are pole-free on the real axis; see Claeys, Kuijlaars and Vanlessen \cite{Claeys2008} and Dai and Hu \cite{Dai:Hu2017}. It is very interesting to note that these pHM and sHM solutions for the inhomogeneous PII play an important role in the study of nematic liquid crystals; see \cite{Clerc2017, Troy2018}. Recently, some novel solutions similar to the AS and HM solutions are obtained by Fornberg and Weideman \cite{Fornberg2014}. They are no longer pole-free but have finitely many poles on the real axis. We will discuss them in the coming section.

\subsection{Quasi-Ablowitz-Segur and quasi-Hastings-McLeod solutions}

It is a well-known fact that, PII transcendents for different parameters $\a$ are related to each other through the following B\"{a}cklund transformation:
\begin{equation}\label{B-transfmn}
u(x;\a)=-u(x;\a-1)+\frac{2\a-1}{2u^{2}(x;\a-1)-2u'(x;\a-1)+x};
\end{equation}
see \cite{Clarkson:NIST}. From the previous section, we know that the AS and sHM solutions exist only when $|\a| < 1/2$. Applying the B\"{a}cklund transformation to the AS and sHM solutions, we will get solutions for $|\a| > 1/2$. It is very interesting to see that the asymptotic behaviors as $x \to \pm \infty$ are reserved under the B\"{a}cklund transformation \eqref{B-transfmn}. The only difference is that, due to properties of the denominator
\begin{equation}\label{f-def}
f(x;\a):=2u^{2}(x;\a)-2u'(x;\a)+x,
\end{equation}
the solutions after the B\"{a}cklund transformation may not be pole-free on the real line. Fornberg and Weideman first observe such kind of solutions and name them the \emph{quasi-Ablowitz-Segur} (qAS) and \emph{quasi-Hastings-McLeod} (qHM) solutions; see \cite[Sec. 4.3]{Fornberg2014}.

Let us define the qAS and qHM solutions with more details below. Due to the symmetry relation \eqref{sym-relation}, we may assume $\a >0$.

\bigskip
\noindent \textbf{qAS solutions: $\a \in (n-1/2,n+1/2),$ $n \in \mathbb{N}$.}
\medskip

Let $k_{\a}$ be a real parameter and $k_{\a} \in (-|\cos \pi \a|, |\cos \pi \a|)$. The qAS solution $u_{\textrm{qAS}}(x;\a)$ is a one-parameter family of solutions of inhomogeneous PII \eqref{PII-def}, which satisfies the asymptotic behaviors in \eqref{asy-pos-AS} and \eqref{asy-neg-AS}, as well as the connection formulas \eqref{d-k} and \eqref{phi-k}.

\medskip
\noindent \textbf{qHM solutions: $\a \in (n-1/2,n+1/2),$ $n \in \mathbb{N}$.}
\medskip

The qHM solutions $u_{\textrm{qHM}}(x;\a)$ of the inhomogeneous PII \eqref{PII-def} have the following asymptotic behaviors
\begin{eqnarray}
    u_{\textrm{qHM}}(x;\a)&=& B(x; \alpha) - \cos(\pi \a) \Ai(x)(1+O(x^{-3/4})), \qquad \textrm{as} \,\, x \to +\infty, \label{asy-pos-qHM} \\
    u_{\textrm{qHM}}(x;\a)&=&-\sqrt{\frac{-x}{2}}- \frac{\a}{2x} +O(-x)^{-3/2}, \hspace{2.5cm}\quad \textrm{as}\,\, x\to -\infty, \label{asy-neg-qHM}
\end{eqnarray}
where the series $B(x; \alpha)$ is given in \eqref{series-B}.

\begin{rmk}\label{qHM-pHM}
The qHM solutions distinguish themselves from the pHM solutions in \eqref{pHM-sigma-value} for $\a > 1/2$ by having different signs in their asymptotics as $x \to \pm \infty$.
\end{rmk}

\begin{rmk}\label{HM-qHM}
If one applies the B\"{a}cklund transformation to the HM solutions $u_{\textrm{HM}}(x;\a)$ (including all the pHM, sHM and qHM solutions) to get a solution $u(x;\a+1)$, it is straightforward to verify that the leading term of the  asymptotics at $-\infty$ is still $\sgn(\a) \sqrt{-x/2}$ for $\a \neq 0$; while the leading term of the asymptotics at $+\infty$ becomes $(\a+1)/x$. Therefore, because the asymptotics of $u(x;\a+1)$ as $x\to \pm \infty$ keep the same sign for $\a>0$, we have
  \begin{equation} \label{HM-transform1}
  \begin{array}{l}
    u_{\textrm{pHM}}(x;\a) \mapsto u_{\textrm{pHM}}(x;\a+1), \\
     u_{\textrm{sHM}}(x;\a) \textrm{ and } u_{\textrm{qHM}}(x;\a) \mapsto u_{\textrm{qHM}}(x;\a+1),
  \end{array}
        \qquad \textrm{for } \a > 0.
  \end{equation}
  Since the term $\a /x \mapsto(\a+1)/x$ changes sign when $-1/2<\a<0$, we obtain
  \begin{equation}
       u_{\textrm{pHM}}(x;\a) \mapsto u_{\textrm{qHM}}(x;\a+1), \quad u_{\textrm{sHM}}(x;\a) \mapsto u_{\textrm{pHM}}(x;\a+1) \quad \textrm{for } -1/2<\a<0.
  \end{equation}
  For the case $\a = 0$, if we choose the following pHM solution with $\sigma = -1$ in \eqref{pHM-sigma-value}
  \begin{equation} \label{HMsol-1}
    u_{\textrm{pHM}}(x;0)\sim - \Ai(x), \quad \textrm{as } x \to +\infty \quad \textrm{and} \quad u_{\textrm{pHM}}(x;0)\sim - \sqrt{-x/2}, \quad \textrm{as } x \to -\infty,
  \end{equation}
  then the B\"{a}cklund transformation \eqref{B-transfmn} gives us
  \begin{equation} \label{HM-transform3}
    u(x;1)\sim 1/x, \quad \textrm{as } x \to +\infty \quad \textrm{and} \quad u(x;1)\sim - \sqrt{-x/2}, \quad \textrm{as } x \to -\infty,
  \end{equation}
  which is $u_{\textrm{qHM}}(x;1)$. Of course, if we put $\sigma = 1$, we get $u_{\textrm{pHM}}(x;1)$.
\end{rmk}

\begin{rmk}
  When applying the B\"{a}cklund transformation to get the qAS solutions, we have
  \begin{equation} \label{BT-coeff}
    \a \mapsto \a +1 \qquad \textrm{and} \qquad k_{\a} \mapsto k_{\a + 1} \ \textrm{ with } k_{\a+1} = -k_{\a},
  \end{equation}
  where $k_\a$ is given in \eqref{asy-pos-AS}. Similar sign change for the coefficient of the $\Ai(x)$ term also occurs while obtaining the qHM solutions, where $\cos(\pi \a)$ is mapped to $\cos(\pi (\a+1))$ in \eqref{asy-pos-HM} and \eqref{asy-pos-qHM}.
\end{rmk}

It is interesting to note that any qAS and qHM solutions defined above admit a B\"{a}cklund transformation \eqref{B-transfmn} with the
Ablowitz-Segur or Hastings-McLeod solution as the seed solution.
To see this, one can study the B\"{a}cklund transformation \eqref{B-transfmn} through Riemann-Hilbert (RH) problems; see Fokas et al. \cite[Sec. 6.1]{Fokas2006}. The AS(qAS) solutions correspond to the following special Stokes multipliers
\begin{equation}
  s_1 = - \sin(\pi \a) - i k_{\a}, \quad  s_2 =0, \quad  s_3 = - \sin(\pi \a) + i k_{\a}, \quad k_{\a} \in (-|\cos \pi \a|, |\cos \pi \a|)
\end{equation}
see \cite[(2.20)]{Dai:Hu2017}. When we change the parameters as in \eqref{BT-coeff}, the Stokes multipliers become $s_{k} \mapsto -s_{k}$, $k = 1,3$. Studying the difference between the two associated RH problems, we get the B\"{a}cklund transformation \eqref{B-transfmn}. Moreover, once the Stokes multipliers are fixed for the qAS solutions $u_{\textrm{qAS}}(x;\a)$, the asymptotics in \eqref{asy-pos-AS} and \eqref{asy-neg-AS} can be derived by using Deift-Zhou nonlinear steepest method uniquely; see for example \cite{Dai:Hu2017,Deift1995,Its2003,Kapaev2004}. Similar arguments also work for the qHM solutions $u_{\textrm{qHM}}(x;\a)$. Since we focus on the pole properties of the qAS and qHM solutions, we will not go into the detailed RH analysis in this paper.

\subsection{Our main results}

We will prove the following results about the poles of qAS and qHM solutions.

\begin{thm} \label{thm-pole-number}
  For $\a\in (n-\frac{1}{2}, n+\frac{1}{2})$ with $n \in \mathbb{N}$, let the qAS solutions $u_{\textrm{qAS}}(x;\a)$ and qHM solutions $u_{\textrm{qHM}}(x;\a)$ be the solution of PII \eqref{PII-def} with asymptotic behaviors given in \eqref{asy-pos-AS}-\eqref{asy-neg-AS} and \eqref{asy-pos-qHM}-\eqref{asy-neg-qHM}, respectively. Then, the qAS and qHM solutions have $n$ real poles.
\end{thm}
\begin{rmk}
  Pole numbers of the qAS and qHM solutions on the real line have been  predicted by Fornberg and Weideman based on the numerical computations in \cite{Fornberg2014}. In the past a few years, Fornberg and Weideman \cite{Fornberg2011,Fornberg2015} have successfully developed the pole field solver (PFS) to compute the Painlev\'e transcendents in the complex plane efficiently and accurately. Recently, they further extend the PFS to study multivalued Painlev\'e transcendents on their Riemann surfaces; see Fasondini, Fornberg and Weideman \cite{Fas:Forn:Weid2017} .
\end{rmk}

Besides the pole numbers, our analysis gives more properties about the poles  on the real axis. It is well-known that the PII transcendents are meromorphic functions whose poles are all simple with residue $\pm 1$; see Gromak et al. \cite[Sec. 2]{Gromak2002}. Our second result shows the dynamics of these poles with respect to the parameter $\a$.

\begin{thm}\label{main-thm-AS}
For $\a\in (n-\frac{1}{2}, n+\frac{1}{2})$ with $n \in \mathbb{N}$, let $p_{i,\pm1}(\a)$ be the $i$-th real pole of $u_{\textrm{qAS}}(x;\a)$ or $u_{\textrm{qHM}}(x;\a)$ counting from the negative real axis, where the subscript $\pm 1$ indicates the residue of the pole is 1 or $- 1$. Then, the poles $p_{i,\pm1}(\a)$ satisfy the following properties:
\begin{itemize}
 \item[(a)] The residue of the smallest pole must be 1. Moreover, it is strictly decreasing with respect to $\a$, i.e.,
 \begin{equation}
   p_{1,+1}(\a) > p_{1,+1}(\a+1) > p_{1,+1}(\a+2) > \cdots .
 \end{equation}

 \item[(b)] The poles with residue $\pm 1$ interlace on the real axis, that is,
 \begin{equation}
 \begin{split}
   &p_{1, +1}(\a)<p_{2, -1}(\a)< \cdots < p_{n, +1}(\a), \quad \textrm{if $n$ is odd,} \\
   &p_{1, +1}(\a)<p_{2, -1}(\a)<\cdots<p_{n, -1}(\a)  , \quad \textrm{if $n$ is even.}
 \end{split}
 \end{equation}

 \item[(c)] All poles of $u(x;\a)$ with residue $+1$ become poles of $u(x;\a+1)$ with residue $-1$ via the B\"{a}cklund transformation, i.e. $p_{i+1, -1}(\a+1) = p_{i, +1}(\a)$; while all poles of $u(x;\a)$ with residue $-1$ are regular points of $u(x;\a+1)$.

 \item[(d)] The residue of the largest pole is $1$ and $-1$ when $n$ is odd and even, respectively. They are increasing with respect to $\a$
 \begin{equation}
  \begin{split}
       & p_{n,+1}(\a) = p_{n+1,-1}(\a+1) < p_{n+2,+1}(\a+2) = p_{n+3,-1}(\a+3) < \cdots , \quad \textrm{if $n$ is odd}, \\
       & p_{n,-1}(\a) < p_{n+1,+1}(\a+1) = p_{n+1,-1}(\a+2) < \cdots , \hspace{3.1cm}\quad \textrm{if $n$ is even}.
     \end{split}
 \end{equation}
\end{itemize}
\end{thm}

The properties in the above theorems can be summarized in the following figure.
\begin{figure}[h]
 \begin{center}
   \includegraphics[width=15cm]{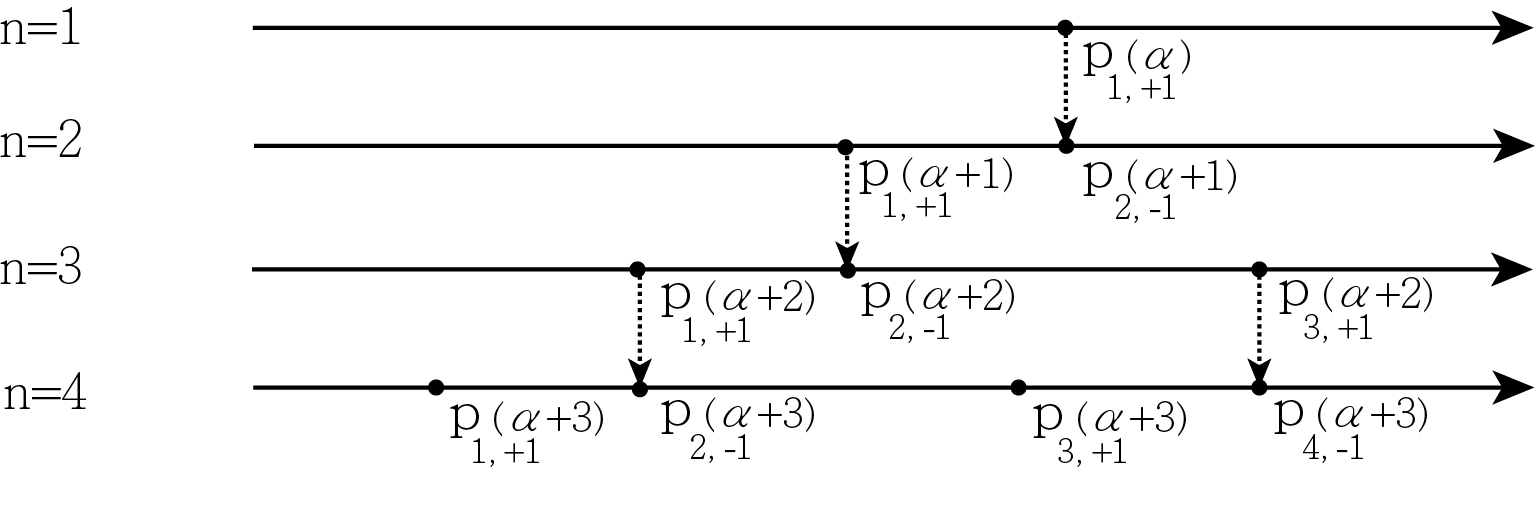}
  \end{center}
 \caption{The poles locations of qAS and qHM solutions of PII on $\mathbb{R}$.}\label{poles}
\end{figure}

The rest of this paper is arranged as follows. In Section \ref{sec-pre-proof}, some properties for the general PII transcendent $u(x;\a)$ and the function $f(x;\a)$ in \eqref{f-def} are provided. Then, in Section \ref{sec-main-proof}, we prove our main results by mathematical induction.

\section{Properties of the general PII transcendents}\label{sec-pre-proof}

First, let us derive some relations between a general solution $u(x;\a)$ and the denominator $f(x;\a)$ in the B\"{a}cklund transformation \eqref{B-transfmn}.

\begin{lemma} \label{lemma-u&f}
  The functions $u(x;\a)$ and $f(x;\a)$ satisfy the following relations:
  \begin{itemize}
    \item[(i)] If $u(x;\a)$ is continuous and differentiable at $x_0$ such that $f(x_0;\a)=0$, then $f'(x_0;\a) = 2\a +1$.

    \item[(ii)] In any interval where $u(x;\a)$ is continuous, $f(x;\a)$ has at most one simple zero when $\a \neq -1/2$.
  \end{itemize}
\end{lemma}
\begin{proof}
From the definition of $f(x;\a)$ in \eqref{f-def}, we have
\begin{equation*}
  f'(x;\a) = 4 u'(x;\a) u(x;\a ) - 2 u''(x;\a) + 1.
\end{equation*}
As $u(x;\a )$ satisfies the PII equation \eqref{PII-def}, it follows from the above formula
\begin{equation}\label{f-der}
f'(x;\a)=-2u(x;\a)f(x;\a)+2\a+1.
\end{equation}
This immediately gives us part (i) of the lemma.

We will prove the second part by contradiction. Suppose that there are two adjacent zeros $x_0$ and $\tilde{x}_0$ of $f(x;\a)$ in the interval $I$ where $u(x;\a)$ is continuous. According to the definition of $f(x;\a)$ in \eqref{f-def}, we know $f(x;\a)$ is continuous and differentiable in $I$.
If $\a \neq -1/2$, then the zeros of $f(x;\a-1)$ must be simple, since
\begin{equation}
f'(x;\a)= 2\a+1 \neq 0.
\end{equation}
Thus, we have
\begin{equation*}
f(x_0;\a)=f(\tilde{x}_0;\a)=0 \quad \textrm{and} \quad f'(x_0;\a)f'(\tilde{x}_0;\a)<0.
\end{equation*}
However, part (i) of the lemma also tells us
\begin{equation}
f'(x_0;\a)=f'(\tilde{x}_0;\a)= 2\a+1,
\end{equation}
which yields a contradiction. Therefore, $f(x;\a)$ has at most one zero in $I$.
\end{proof}

Next, we study the pole properties under the B\"{a}cklund transformation.

\begin{lemma} \label{lemma-bt-u}
Let the solutions $u(x;\a -1)$ and $u(x;\a)$ be related via the B\"{a}cklund transformation \eqref{B-transfmn}. Then, we have
\begin{itemize}
  \item[(i)] The poles of $u(x;\a -1)$ with residue $+1$ are poles of $u(x;\a)$ with residue $-1$.

  \item[(ii)] If $p_{-1}$ is a pole of $u(x;\a -1)$ with residue $-1$, then $p_{-1}$ is a regular point of $u(x;\a)$.
\end{itemize}
\end{lemma}

\begin{proof}
To prove part (i), let us assume $p_{+1}$ is a pole of $u(x; \a -1)$ with residue $+1$. Then we have the following expansion near $p_{+1}$
\begin{equation}\label{u(a-1)-pole-res(+1)}
u(x;\a-1)= \frac{1}{x-p_{+1}}-\frac{p_{+1}}{6}(x-p_{+1})+\frac{\a-2}{4}(x-p_{+1})^2+O((x-p_{+1})^3), \qquad \textrm{as } x\to p_{+1}.
\end{equation}
Using the definition of $f(x;\a)$ in \eqref{f-def} and the above formula, we can see that $p_{+1}$ is a double pole of $f(x;\a -1)$:
\begin{equation}\label{f(a-1)-pole-res(+1)}
f(x;\a -1)=\frac{4}{(x-p_{+1})^2}+O((x-p_{+1})^{-3}), \qquad \textrm{as } x\to p_{+1}.
\end{equation}
From the B\"{a}cklund transformation \eqref{B-transfmn}, it is easy to obtain
\begin{equation}
u(x; \a) = -\frac{1}{x-p_{+1}}+\frac{p_{+1}}{6}(x-p_{+1})+ \frac{\a+1}{4}(x-p_{+1})^2+O((x-p_{+1})^3), \qquad \textrm{as } x\to p_{+1}.
\end{equation}
Therefore, $p_{+1}$ is a pole of $u(x;\a)$ with residue $-1$.

Similarly, we can prove part (ii). If $p_{-1}$ is a pole of $u(x; \a -1)$ with residue $-1$, the following expansion near $p_{-1}$ holds:
\begin{equation}
u(x;\a-1)= \frac{-1}{x-p_{-1}}+\frac{p_{-1}}{6}(x-p_{-1})+\frac{\a}{4}(x-p_{-1})^2+O((x-p_{-1})^3), \qquad \textrm{as } x\to p_{-1}.
\end{equation}
Then, $p_{-1}$ is a simple zero of $f(x;\a -1)$:
\begin{equation}\label{f(a-1)-zero-res(-1)}
f(x;\a -1)=(-2\a+1)(x-p_{-1})+O((x-p_{-1})^2), \qquad \textrm{as } x\to p_{-1}.
\end{equation}
As a consequence, the second term in the B\"{a}cklund transformation \eqref{B-transfmn} induces a simple pole with residue $-1$, which cancels the pole contribution from the first term. Thus, $p_{-1}$ is a regular point of $u(x;\a)$.
\end{proof}


\section{Proof of Theorems \ref{thm-pole-number} and \ref{main-thm-AS}}\label{sec-main-proof}
\subsection{Properties of real poles of qAS solutions}

For $\a\in (n-\frac{1}{2}, n+\frac{1}{2})$, we will study the qAS solutions for $n=1,2,3$, which possess all properties listed in Theorems \ref{thm-pole-number} and \ref{main-thm-AS}. Then, we will prove our results by mathematical induction for all $n \in \mathbb{N}$. First, we show that, for $\a\in (\frac{1}{2}, \frac{3}{2})$, the qAS solutions have only one pole on the real line.
\begin{prop}\label{thm-one-pole}
For $\a\in (\frac{1}{2}, \frac{3}{2})$, the qAS solutions of \eqref{PII-def} have only one pole on the real line with residue $+1$.
\end{prop}

\begin{proof}
For $\a\in (\frac{1}{2}, \frac{3}{2})$, $u_{\textrm{AS}}(x;\a-1)$ is the pole-free AS solution on the real line. To get the unique pole of $u_{\textrm{qAS}}(x;\a)$, it is enough to show that $f(x;\a-1)$ in the B\"{a}cklund transformation \eqref{B-transfmn} has only one zero on the real line. Recalling the asymptotics of $u_{\AS}(x;\a-1)$ in \eqref{asy-pos-AS} and \eqref{asy-neg-AS}, we have from \eqref{f-def}
\begin{equation}
f(x;\a-1)\sim x, \qquad \textrm{as}\,\, x\to \pm \infty.
\end{equation}
Moreover, since $u_{\AS}(x; \a-1)$ is smooth on the real line, $f(x;\a-1)$ is also continuous on the real line. The above asymptotics imply that $f(x;\a-1)$ has zeros on the real line.

According Lemma \ref{lemma-u&f}, there is a unique point $x_1$ such that $f(x_1; \a-1)=0$ and $f'(x_1,\a-1)=2\a -1$. Then, using the B\"{a}cklund transformation \eqref{B-transfmn}, we conclude that $u_{\textrm{qAS}}(x;\a)$ has only one pole on the real line with residue $+1$.
\end{proof}

\begin{prop}\label{thm-two-poles}
For $\a\in (\frac{3}{2}, \frac{5}{2})$, the qAS solutions of \eqref{PII-def} have two poles on the real line with residues $\pm 1$. Moreover, we have $p_{1,+1}(\a) < p_{2,-1}(\a)$.
\end{prop}

\begin{proof}
For $\a\in (\frac{3}{2}, \frac{5}{2})$, let $x_{1, +1}$ be the unique pole of $u_{\textrm{qAS}}(x;\a-1)$ with residue $+1$. Then, from part (i) in Lemma \ref{lemma-bt-u}, $x_{1, +1}$ is the pole of $u_{\textrm{qAS}}(x;\a)$ with residue $-1$. To find the other pole of $u_{\textrm{qAS}}(x;\a)$, we make use of the behavior of $u_{\textrm{qAS}}(x;\a-1)$ near $x_{1, +1}$ given in \eqref{u(a-1)-pole-res(+1)}.
Combining the formulas \eqref{asy-pos-AS}-\eqref{asy-neg-AS}, we have from \eqref{f-def}
\begin{equation}\label{f(a-1)-infty-asy-1}
f(x;\a-1)\sim x,\qquad \textrm{as}\,\, x \to \pm\infty \quad \textrm{and} \quad f(x;\a-1)\to +\infty, \qquad \textrm{as}\,\, x \to x_{1, +1}.
\end{equation}
Based on the similar analysis in Proposition \ref{thm-one-pole}, we get that $f(x;\a-1)$ has only one zero $x_0$ in $(-\infty, x_{1, +1})$ and $f'(x_0; \a-1)= 2\a-1$. Therefore, $x_0$ is the unique pole of $u_{\textrm{qAS}}(x;\a)$ in $(-\infty, x_{1, +1})$ with residue $+1$.

Finally, to show that there is no other poles, we verify that $u_{\textrm{qAS}}(x;\a)$ has no pole in $(x_{1, +1}, +\infty)$. Otherwise, $f(x;\a-1)$ must have zeros in $(x_{1, +1}, +\infty)$. Due to the asymptotics in \eqref{f(a-1)-infty-asy-1}, $f(x;\a-1)$ tends to $+\infty$ at both endpoints of the interval $(x_{1, +1}, +\infty)$.  Then, $f(x;\a-1)$ has at least two zeros in $(x_{1, +1}, +\infty)$, where we use the fact that all zeros of $f(x;\a-1)$ are simple. So, we arrive at a contradiction with part (ii) of Lemma \ref{lemma-u&f}.

This completes the proof of our proposition.
\end{proof}

\begin{prop}\label{thm-three-poles}
For $\a\in (\frac{5}{2}, \frac{7}{2})$, the qAS solutions of \eqref{PII-def} have three poles on the real line with residues $\pm 1$. Moreover, we have $p_{1,+1}(\a) < p_{2,-1}(\a)< p_{3,+1}(\a)$.
\end{prop}

\begin{proof}
For $\a\in (\frac{5}{2}, \frac{7}{2})$, let $x_{1, +1}<x_{2, -1}$ be the two poles of $u_{\textrm{qAS}}(x;\a-1)$ with residues $+1$ and $-1$, respectively. Then, from part (i) in Lemma \ref{lemma-bt-u}, $x_{1, +1}$ is the pole of $u_{\textrm{qAS}}(x;\a)$ with residue $-1$. Using the similar analysis in Proposition \ref{thm-two-poles}, it is easy to show that there is a unique pole of $u_{\textrm{qAS}}(x;\a)$ in $(-\infty, x_{1, +1})$ with residue $+1$. To find the last pole of $u_{\textrm{qAS}}(x;\a)$, let us study the property of $f(x;\a-1)$ in $(x_{1, +1}, +\infty)$. Using a similar computation in \eqref{f(a-1)-infty-asy-1}, we have
\begin{eqnarray}
f(x;\a-1)\to  +\infty, \quad  \textrm{as } x\to x_{1, +1}, \quad \textrm{and} \quad
f(x;\a-1)\to +\infty, \quad  \textrm{as } x\to +\infty.
\end{eqnarray}
Note that, although $x_{2, -1} \in (x_{1, +1}, +\infty)$, it is a regular point of $u_{\textrm{qAS}}(x;\a)$; see part (ii) of Lemma \ref{lemma-bt-u}. Moreover, we have from \eqref{f(a-1)-zero-res(-1)}
\begin{equation}
f(x_{2, -1};\a-1) = 0 \quad \textrm{and} \quad f'(x_{2, -1};\a-1)=-2\a+1<0.
\end{equation}
As $f(x;\a-1)$ is continuous and differentiable on $(x_{1, +1}, +\infty)$, the above two formulas yield there must exist $x_3>x_{2, -1}$ such that $f(x_3;\a-1)=0$ and $f'(x_3;\a-1)>0$. Note that $u_{\textrm{qAS}}(x;\a-1)$ is continuous on $(x_{2, -1}, +\infty)$. According to Lemma \ref{lemma-u&f}, $x_3$ is the unique zero of $f(x;\a-1)$ in $(x_{2, -1}, +\infty)$ and $f'(x_3;\a-1)= 2\a -1$. Therefore, $x_3$ must be a pole of $u_{\textrm{qAS}}(x;\a)$ with residue $+1$. By Lemma \ref{lemma-u&f} again, $u_{\textrm{qAS}}(x;\a)$ has no pole in $(x_{1, +1}, x_{2, -1})$.

This completes the proof of our proposition.
\end{proof}

Finally, we prove the statements involving the qAS solutions in Theorems \ref{thm-pole-number} and \ref{main-thm-AS} by mathematical induction.

\medskip
\noindent\emph{Proof of Theorem \ref{thm-pole-number} and \ref{main-thm-AS}.} The above three propositions indicate that Theorems \ref{thm-pole-number} and \ref{main-thm-AS} are true for $n=1,2,3$. Assume the results also hold for $n= m$, let us consider the case for $m+1$. We denote the poles of $u_{\textrm{qAS}}(x;\a)$ and $u_{\textrm{qAS}}(x;\a+1)$ by $p_{i,\pm1}(\a)$ and $\widetilde{p}_{j,\pm1}(\a+1)$, respectively.

As the residue of the smallest pole of $u_{\textrm{qAS}}(x;\a)$ is 1, following similar analysis in Proposition \ref{thm-two-poles}, there exists a unique pole $x_0$ of $u_{\textrm{qAS}}(x;\a+1)$ in $(-\infty, p_{1,+1}(\a))$ with residue $1$. This proves part (a) of Theorem \ref{main-thm-AS}.

Let $p_{k,+1}(\a)< p_{k+1,-1}(\a) < p_{k+2,+1}(\a)$ be three consecutive poles of $u_{\textrm{qAS}}(x;\a)$. According to Lemma \ref{lemma-bt-u}, they are mapped to $\widetilde{p}_{j, -1}(\a+1)< \eta_0 < \widetilde{p}_{j+2, -1}(\a+1)$, where $\eta_0$ is the zero of $f(x;\a)$ and a regular point of $u_{\textrm{qAS}}(x;\a+1)$. Since $p_{k,+1}(\a)$ and $p_{k+2,+1}(\a)$ are poles of $u_{\textrm{qAS}}(x;\a)$ with residue $+1$, we have from \eqref{f(a-1)-pole-res(+1)}
\begin{equation}
  f(x;\a)\to  +\infty, \quad  \textrm{as } x\to p_{k,+1}(\a) \textrm{ and } x \to  p_{k+2,+1}(\a).
\end{equation}
Using the similar analysis in Proposition \ref{thm-three-poles}, there exists a unique point $x^*\in (p_{k+1,-1}(\a), p_{k+2,+1}(\a))$ such that $f(x^*;\a)=0$ and $f'(x^*;\a)= 2\a +1>0$. This shows that $x^*$ is the unique pole of $u_{\textrm{qAS}}(x;\a+1)$ with residue $+1$ in $(\eta_0, \widetilde{p}_{j+2, -1}(\a+1))$. Therefore, we obtain three consecutive poles of $u_{\textrm{qAS}}(x;\a+1)$: $\widetilde{p}_{j, -1}(\a+1)< \widetilde{p}_{j+1, +1}(\a+1) < \widetilde{p}_{j+2, -1}(\a+1)$ with $\widetilde{p}_{j+1, +1}(\a+1) = x^*$. Thus, we prove the interlacing property of poles with residue $\pm 1$, i.e., part (b) of Theorem \ref{main-thm-AS}. The part (c) of Theorem \ref{main-thm-AS} is indeed Lemma \ref{lemma-bt-u}.

The proof of the pole numbers and the largest pole also follows from the arguments above. Lemma \ref{lemma-bt-u} and arguments in the previous paragraph imply that, for the interval $I= [p_{k_1,+1}(\a), p_{k_2,+1}(\a)] $ with any $k_1<k_2$, $u_{\textrm{qAS}}(x;\a+1)$ has the same number of poles as $u_{\textrm{qAS}}(x;\a)$ in $I$. When $m$ is odd, as the residues of both the smallest and largest pole of $u_{\textrm{qAS}}(x;\a)$ are 1, $u_{\textrm{qAS}}(x;\a+1)$ has $m$ poles in $[p_{1,+1}(\a), p_{m,+1}(\a)]$ with ${p}_{2,-1}(\a+1)=p_{1,+1}(\a)$ and ${p}_{m+1,-1}(\a+1)=p_{m,+1}(\a)$. Recalling that there is one more pole with residue 1 in  $(-\infty, p_{1,+1}(\a))$, then $u_{\textrm{qAS}}(x;\a+1)$ has $m+1$ poles with the largest pole ${p}_{m+1,-1}(\a+1)=p_{m,+1}(\a)$. When $m$ is even, the situation is similar. Now, $u_{\textrm{qAS}}(x;\a+1)$ has $m$ poles in $(-\infty , p_{m-1,+1}(\a)]$. By the similar analysis in Proposition \ref{thm-three-poles}, there is one more pole with residue $+1$ in $(p_{m,-1}(\a), + \infty)$. Thus, $u_{\textrm{qAS}}(x;\a+1)$ has $m+1$ poles with the largest pole ${p}_{m+1,+1}(\a+1)>p_{m,-1}(\a)$. This proves Theorem \ref{thm-pole-number} and part (d) of Theorem \ref{main-thm-AS}.

Then, we finish proof of the results involving the qAS solutions in Theorem \ref{thm-pole-number} and \ref{main-thm-AS}.
 \hfill $\Box$

\subsection{Properties of real poles of qHM solutions}


Using the same idea in the previous section, we first show that, for $\a\in (\frac{1}{2}, \frac{3}{2})$, the qHM solutions corresponding to the asymptotics in \eqref{asy-pos-qHM} and \eqref{asy-neg-qHM} have only one pole on the real line.
\begin{prop}\label{thm-one-pole-HM}
For $\a\in (\frac{1}{2}, \frac{3}{2})$, the qHM solutions of \eqref{PII-def} have only one pole on the real line with residue $+1$.
\end{prop}
\begin{proof}
For $\a\in (\frac{1}{2}, \frac{3}{2})$, the qHM solutions $u_{\textrm{qHM}}(x;\a)$ are transform from the pHM and sHM solutions (cf. \eqref{HM-transform1}-\eqref{HM-transform3}), which are continuous on the real line. Using the asymptotics of these solutions in \eqref{asy-pos-HM}, \eqref{asy-neg-HM} and \eqref{HMsol-1}, we obtain from \eqref{f-def}
\begin{equation}
f(x;\a-1)\sim x, \quad \textrm{as}\,\, x\to  \infty \quad \textrm{and} \quad f(x;\a-1)\to \frac{1-2\a}{2}\left(\frac{-x}{2}\right)^{-1/2}, \quad \textrm{as}\,\, x\to -\infty.
\end{equation}
As $f(x;\a-1)$ is continuous on the real line, it has zeros on the real line. By the similar argument in Proposition \ref{thm-one-pole}, we know that $u_{\textrm{qHM}}(x;\a)$ has only one pole on the real line with residue $+1$.
\end{proof}

Following similar analysis in the previous section, it is easy to see that Propositions \ref{thm-two-poles} and \ref{thm-three-poles} also hold for the qHM solutions. Using mathematical induction again, we prove the statements involving the qHM solutions in Theorems \ref{thm-pole-number} and \ref{main-thm-AS}.

\section*{Acknowledgements}

The authors were partially supported by grants from the Research Grants Council of the Hong
Kong Special Administrative Region, China (Project No. CityU 11300814, CityU 11300115, CityU 11303016).


\begin{thebibliography}{99}

\bibitem{Ablowitz1977-asymptotic}
M.~J. Ablowitz and H.~Segur,
\newblock Asymptotic solutions of the {K}orteweg-de{V}ries equation,
\newblock {\em Studies in Appl. Math.}, \textbf{57} (1976/77), no. 1, 13--44.



\bibitem{Bas:Cla:Law:McL1998} A. P. Bassom, P. A. Clarkson, C. K. Law and J. B. McLeod, Application of uniform asymptotics to the second Painlev\'e transcendent, \emph{Arch. Rational Mech. Anal.},
\textbf{143} (1998), no. 3, 241--271.


\bibitem{Claeys2008}
T. Claeys, A. B. J. Kuijlaars and M. Vanlessen,
\newblock Multi-critical unitary random matrix ensembles and the general {P}ainlev\'e II equation,
\newblock {\em Ann. of Math.},  \textbf{168} (2008), no. 2, 601--641.


\bibitem{Clarkson:NIST} P. A. Clarkson, Painlev\'e transcendents, \emph{NIST Handbook of Mathematical Functions}, 723--740, U.S. Dept. Commerce, Washington, DC, 2010.

\bibitem{Clarkson1988}
P.~A. Clarkson and J.~B. McLeod,
\newblock A connection formula for the second {P}ainlev\'e transcendent,
\newblock {\em Arch. Rational Mech. Anal.},  \textbf{103} (1988), no. 2, 97--138.

\bibitem{Clerc2017}
M. G. Clerc, J. D. D\'{a}vila, M. Kowalczyk, P. Smyrnelis and E. Vidal-Henriquez,
\newblock Theory of light-matter interaction in nematic liquid crystals and the second Painlev\'{e} equation,
\newblock {\em  Calc. Var. Partial Differential Equations}, \textbf{56} (2017), no. 4, 56--93.


\bibitem{Dai:Hu2017}
D. Dai and W. Y. Hu,
\newblock Connection formulas for the Ablowitz-Segur solutions of the inhomogeneous {P}ainlev\'e II equation,
\newblock {\em Nonlinearity},  \textbf{30} (2017), 2982--3009.

\bibitem{Deift1995}
P. A. Deift and X.~Zhou,
\newblock Asymptotics for the {P}ainlev\'e {II} equation,
\newblock {\em Comm. Pure Appl. Math.},  \textbf{48} (1995), no. 3, 277--337.

\bibitem{Fas:Forn:Weid2017} M. Fasondini, B. Fornberg and J. A. C. Weideman, Methods for the computation of the multivalued Painlev\'e transcendents on their Riemann surfaces, \emph{J. Comput. Phys.}, \textbf{344} (2017), 36--50.

\bibitem{Fokas2006}
A.~S. Fokas, A.~R. Its, A.~A. Kapaev and V.~Y. Novokshenov,
\newblock {\em Painlev\'e transcendents: The Riemann-Hilbert approach}, Math. Surv. Monog., vol. 128,
\newblock Amer. Math. Soc., Providence, RI, 2006.

\bibitem{Fornberg2011} B. Fornberg, J. A. C. Weideman, A numerical methodology for the Painlev\'e equations, \emph{J. Comput.
Phys.}, \textbf{230} (2011), 5957--5973.

\bibitem{Fornberg2014}
B.~Fornberg and J. A. C. Weideman, A computational exploration of the second Painlev\'e equation, \emph{Found. Comput. Math.}, \textbf{14} (2014), no. 5, 985--1016.

\bibitem{Fornberg2015} B.~Fornberg and J. A. C. Weideman, A computational overview of the solution space of the imaginary Painlev\'e II equation,
\emph{Phys. D}, \textbf{309} (2015), 108--118.


\bibitem{Gromak2002}
V.~I. Gromak, I.~Laine and S.~Shimomura,
\emph{Painlev\'e Differential Equations in the Complex Plane}, de Gruyter Studies in Mathematics, 28. Walter de Gruyter \& Co., Berlin, 2002.

\bibitem{HM1980}
S.~P. Hastings and J.~B. McLeod, A boundary value problem associated with the second Painlev\'e transcendent and the Korteweg-de\thinspace Vries equation, \emph{Arch. Rational Mech. Anal.}, \textbf{73} (1980), no. 1, 31--51.

\bibitem{Its2003}
A.~R. Its and A.~A. Kapaev,
\newblock Quasi-linear {S}tokes phenomenon for the second {P}ainlev\'e
  transcendent, \emph{Nonlinearity}, \textbf{16} (2003), no. 1, 363--386.


\bibitem{Kapaev1992} A. Kapaev, Global asymptotics of the second Painlev\'e transcendent,
\emph{Phys. Lett. A}, \textbf{167} (1992), no. 4, 356--362.

\bibitem{Kapaev2004}
A.~A. Kapaev,
\newblock Quasi-linear Stokes phenomenon for the Hastings-McLeod solution of
  the second Painlev{\'e} equation, arXiv:nlin.SI/0411009.

\bibitem{mccoy:Tang1986} B. M. McCoy and S. Tang, Connection formulae for Painlev\'e functions, \emph{Phys. D}, \textbf{18} (1986), no. 1-3, 190--196.

\bibitem{Segur1981}
H.~Segur and M.~J. Ablowitz,
\newblock Asymptotic solutions of nonlinear evolution equations and a
  Painlev{\'e} transcedent,
\newblock {\em  Phys. D}, \textbf{3} (1981), 165--184.

\bibitem{Troy2018}
W. C. Troy,
\newblock The role of Painlev\'{e} II in predicting new liquid crystal self-assembly mechanisms,
\newblock {\em Arch. Rational Mech. Anal.}, \textbf{227} (2018), no. 1, 367--385.

\end{thebibliography}
\end{document}